\documentclass[12pt]{article}
\usepackage{fullpage,amsthm,amssymb,amsmath}
\usepackage{graphicx,algorithmic,algorithm,amsfonts, verbatim}
\usepackage{enumerate}

\newtheorem{theorem}{Theorem}[section]
\newtheorem{lemma}[theorem]{Lemma}
\newtheorem{corollary}[theorem]{Corollary}
\newtheorem{claim}[theorem]{Claim}
\newtheorem{problem}[theorem]{Problem}
\newtheorem{question}[theorem]{Question}

\newcommand\abs[1]{\lvert #1\rvert}

\begin{document}
\title{$(1, k)$-coloring of graphs with girth at least $5$ on a surface}
\author{
Hojin Choi\thanks{
Department of Mathematical Sciences, Korea Advanced Institute of Science and Technology, Daejeon, Republic of Korea.
\texttt{\{hojinchoi,ilkyoo,jjisu,gwsuh91\}@kaist.ac.kr}\newline
All authors are supported by Basic Science Research
  Program through the National Research Foundation of Korea (NRF)
  funded by the Ministry of Science, ICT \& Future Planning
  (2011-0011653).
}
\and
Ilkyoo Choi$^*$\thanks{Corresponding author. }
\and
Jisu Jeong$^*$
\and
Geewon Suh$^*$
}
\date\today
\maketitle
\begin{abstract}
A graph is $(d_1, \ldots, d_r)$-colorable if its vertex set can be partitioned into $r$ sets $V_1, \ldots, V_r$ so that the maximum degree of the graph induced by $V_i$ is at most $d_i$ for each $i\in \{1, \ldots, r\}$.
For a given pair $(g, d_1)$, the question of determining the minimum $d_2=d_2(g; d_1)$ such that planar graphs with girth at least $g$ are $(d_1, d_2)$-colorable has attracted much interest. 
The finiteness of $d_2(g; d_1)$ was known for all cases except when $(g, d_1)=(5, 1)$.
Montassier and Ochem explicitly asked if $d_2(5; 1)$ is finite. 
We answer this question in the affirmative with $d_2(5; 1)\leq 10$; namely, we prove that all planar graphs with girth at least $5$ are $(1, 10)$-colorable. 
Moreover, our proof extends to the statement that for any surface $S$ of Euler genus $\gamma$, there exists a $K=K(\gamma)$ where graphs with girth at least $5$ that are embeddable on $S$ are $(1, K)$-colorable. 
On the other hand, there is no finite $k$ where planar graphs (and thus embeddable on any surface) with girth at least $5$ are $(0, k)$-colorable.

\end{abstract}

\section{Introduction}

Only finite, simple graphs are considered. 
Given a graph $G$, let $V(G)$ and $E(G)$ denote the vertex set and the edge set of $G$, respectively. 
A {\it neighbor} of a vertex $v$ is a vertex adjacent to $v$, and let $N(v)$ denote the set of neighbors of $v$.
The {\it degree} of a vertex $v$, denoted $d(v)$, is $|N(v)|$.
The {\it degree} of a face $f$, denoted $d(f)$, is the length of a shortest boundary walk of $f$. 
A {\it $k$-vertex}, {\it $k^+$-vertex}, {\it $k^-$-vertex} is a vertex of degree $k$, at least $k$, at most $k$, respectively. 
A {\it $k$-face} is a face of degree $k$. 

Given a graph $G$, the chromatic number is the minimum $k$ such that the vertex set $V(G)$ can be partitioned into $k$ parts so that each part induces a graph with maximum degree at most $0$. 
Maybe we can allow some non-zero maximum degree in each color class and obtain a partition with fewer parts than the chromatic number; this notion is known as {\it improper coloring}, and has been widely studied recently.
To be precise, a graph is {\it $(d_1, \ldots, d_r)$-colorable} if its vertex set can be partitioned into $r$ sets $V_1, \ldots, V_r$ so that the maximum degree of the graph induced by $V_i$ is at most $d_i$ for each $i\in \{1, \ldots, r\}$;
in other words, there exists a function $\varphi:V(G)\rightarrow\{1, \ldots, r\}$ where the graph induced by vertices of color $i$ has maximum degree at most $d_i$ for $i\in\{1, \ldots, r\}$.

There are many papers that study $(d_1, \ldots, d_r)$-colorings of sparse graphs resulting in corollaries regarding planar graphs, often with restrictions on the girth. 
The well-known Four Color Theorem~\cite{1977ApHa,1977ApHaKo} is exactly the statement that planar graphs are $(0, 0, 0, 0)$-colorable. 
Cowen, Cowen, and Woodall~\cite{1986CoCoWo} proved that planar graphs are $(2, 2, 2)$-colorable, and Eaton and Hull~\cite{1999EaHu} and \v Skrekovski~\cite{1999Sk} proved that this is sharp by exhibiting a non-$(1, k, k)$-colorable planar graph for each $k$. 
Thus, the problem of improper coloring planar graphs with at least three parts is completely solved.

Naturally, the next line of research is to consider improper colorings of planar graphs with two parts.
The following two questions are attracting much interest. 

\begin{problem}\label{problem-girth}
Given a pair $(d_1, d_2)$, determine the minimum $g=g(d_1,d_2)$ such that every planar graph with girth $g$ is $(d_1, d_2)$-colorable. 
\end{problem}

\begin{problem}\label{problem-d2}
Given a pair $(g, d_1)$, determine the minimum $d_2=d_2(g; d_1)$ such that every planar graph with girth $g$ is $(d_1, d_2)$-colorable. 
\end{problem}

Regarding Problem~\ref{problem-girth}, the special case when $d_1=d_2$ was first considered by Cowen, Cowen, and Woodall~\cite{1986CoCoWo}, who constructed a planar graph that is not $(d, d)$-colorable for each $d$.
This shows that $g(d, d)\geq 4$ for all $d$, and instigated the girth constraint for future research in this area.
\v Skrekovski~\cite{1999Sk,2000Sk} continued the study and obtained some bounds on $g(d, d)$, which were improved by Havet and Sereni~\cite{2006HaSe} and Borodin, Kostochka, and Yancey~\cite{2013BoKoYa}; the current best known bounds are $6\leq g(1, 1)\leq 7$ and $5\leq g(3, 3)\leq g(2, 2)\leq 6$ and $g(d, d)=5$ for $d\geq 4$. 
Note that since $g(d_1+1, d_2+1)\leq g(d_1, d_2+1)\leq g(d_1, d_2)$, we know that $g(d_1, d_2)=5$ whenever $d_1, d_2\geq 4$. 

Values of $g(d_1, d_2)$ that are determined when $\min\{d_1, d_2\}\leq 3$ are $g(0, d_2)=7$ when $d_2\geq 4$ by Borodin and Kostochka~\cite{2014BoKo} and Borodin, Ivanova, Montassier, Ochem, and Raspaud~\cite{2010BoIvMoOcRa}, $g(2, d_2)=5$ when $d_2\geq 6$ by Havet and Sereni~\cite{2006HaSe} and \v Skrekovski~\cite{2000Sk}, and $g(3, d_2)=5$ when $d_2\geq 5$, proved by Choi and Raspaud~\cite{unpub_ChRa} and \v Skrekovski~\cite{2000Sk}.

To our knowledge, the exact value of $g(1, d_2)$ were not known for any value of $d_2$ before this paper. 
Our main result (Theorem~\ref{thm-main}) determines infinitely many values of $g(1, d_2)$; namely, our main result implies that $g(1, d_2)=5$ when $d_2\geq 10$.
These facts are summerized in the following theorem.

\begin{theorem}[\cite{2014BoKo,2010BoIvMoOcRa,2006HaSe,2000Sk,unpub_ChRa}]
If $g(d_1,d_2)$ is the minimum $g$ where every planar graph with girth $g$ is $(d_1, d_2)$-colorable, then 
\begin{itemize}
\item $g(0, d_2)=7$ for $d_2\geq 4$,

\item $g(1, d_2)=5$ for $d_2\geq 10$,

\item $g(2, d_2)=5$ for $d_2\geq 6$,

\item $g(3, d_2)=5$ for $d_2\geq 5$,

\item $g(d_1, d_2)=5$ for $d_1,d_2\geq 4.$
\end{itemize}
\end{theorem}

None of the thresholds on $d_2$ in the previous theorem are known to be tight. 
There has been substantial effort in trying to find the exact value of $g(0, 1)$ by various authors~\cite{2007GlZa,2009BoIv,2011BoKo,2013EsMoOcPi,unpub_KiKoZh}, and the current best bound is $10\leq g(0, 1)\leq 11$ by Kim, Kostochka, and Zhu~\cite{unpub_KiKoZh} and Esperet, Montassier, Ochem, and Pinlou~\cite{2013EsMoOcPi}.
Interestingly, it is known that $g(0, 2)=8$, proved by Montassier and Ochem~\cite{unpub_MoOc} and Borodin and Kostochka~\cite{2014BoKo}.

\bigskip

Regarding Problem~\ref{problem-d2}, we know that $d_2(g; d_1)$ is not finite when either $g\in\{3, 4\}$ by Montassier and Ochem~\cite{unpub_MoOc} or $d_1=0$ and $g\leq 6$ by Borodin, Ivanova, Montassier, Ochem, and Raspaud~\cite{2010BoIvMoOcRa}. 
Various authors~\cite{2010BoIvMoOcRa,2014BoKo,2013BoKoYa,2012BoIvMoRa,2006HaSe,2011BoIvMoRa,unpub_MoOc} conducted research on trying to find the exact value of $d_2(g; d_1)$ for various pairs $(g, d_1)$.
See Table~\ref{table} for the current best known bounds.
Improving any value in the table would be a noteworthy result.

\begin{table}[h]
\begin{center}	
\begin{tabular}{c||c|c|c|c|c}
\mbox{girth} & $(0,d_2)$ & $(1,d_2)$ & $(2,d_2)$ & $(3,d_2)$ & $(4,d_2)$ \\
\hline
3 & $\times$ & $\times$ & $\times$ & $\times$ & $\times$ \\
4 & $\times$ & $\times$ & $\times$ & $\times$ & $\times$ \\
5 & $\times$ & $(1,10)$  & $(2,6)$~\cite{2014BoKo}     &  $(3,5)$~\cite{unpub_ChRa}  & $(4, 4)$~\cite{2000Sk}\\
6 & $\times$~\cite{2010BoIvMoOcRa}& $(1,4)$~\cite{2014BoKo}  & $(2, 2)$~\cite{2006HaSe}    & 			&  		\\
7 & $(0,4)$~\cite{2014BoKo}	& $(1, 1)$~\cite{2013BoKoYa}  & 		      & 			&  		\\
8 & $(0,2)$~\cite{2014BoKo}	& 		    & 		      & 			&  		\\
11 & $(0,1)$~\cite{unpub_KiKoZh}	& 		    & 		      & 			&  		
\end{tabular}
\caption{Table of $d_2(g; d_1)$}
\label{table}
\end{center}
\end{table}

The finiteness of $d_2(g; d_1)$ was known for all pairs $(g, d_1)$, except when $(g, d_1)=(5, 1)$. 
This case was an open question that was explicitly asked in Montassier and Ochem~\cite{unpub_MoOc} and was also mentioned in Choi and Raspaud~\cite{unpub_ChRa}.

\begin{question}[\cite{unpub_MoOc,unpub_ChRa}]\label{question}
Does there exist a finite $k$ where planar graphs with girth at least $5$ are $(1, k)$-colorable?
\end{question}

We answer Question~\ref{question} in the affirmative by proving that $d_2(5; 1)\leq 10$. 
This finishes the long journey of characterizing all pairs $(g, d_1)$ where $d_2(g; d_1)$ is finite. 
Moreover, our proof easily extends to the statement that for any surface $S$ of Euler genus $\gamma$, there exists a $K=K(\gamma)$ where graphs with girth at least $5$ that are embeddable on $S$ are $(1, K)$-colorable. 
This is best possible in the sense that it was already known that there is no finite $k$ where planar graphs with girth at least $5$ are $(0, k)$-colorable~\cite{2010BoIvMoOcRa}.
The following is the precise statement of our main results.

\begin{theorem}\label{thm-main}
Projective planar graphs with girth at least $5$ are $(1, 10)$-colorable. 
\end{theorem}

\begin{corollary}\label{cor-planar}
Planar graphs with girth at least $5$ are $(1, 10)$-colorable. 
\end{corollary}

\begin{theorem}\label{thm-gen}
Given a surface $S$ of Euler genus $\gamma$, every graph with girth at least $5$ that is embeddable on $S$ is $\big(1, K(\gamma)\big)$-colorable where $K(\gamma)=\max\{10, 4\gamma+3\}$.
\end{theorem}

In Section~\ref{section-lemmas}, we will reveal some structure of minimum counterexamples to Theorem~\ref{thm-main}. 
We use discharging, and the discharging rules are laid out in Section~\ref{section-discharging}. 
Finally, we finish the proof of Theorem~\ref{thm-main} and explain how the proof extends to the proof of Theorem~\ref{thm-gen} in Section~\ref{section-proof}.
The ideas of the lemmas used in the proof of Theorem~\ref{thm-gen} are in Section~\ref{section-lemmas}, yet, in order to improve the readability of the paper, we will not explicitly rewrite all the lemmas for the case $\gamma\geq 2$ as they do not add more value.

From now on, assume a graph $G$ is a counterexample to Theorem~\ref{thm-main} with the minimum number of vertices, and fix some embedding of $G$.
Note that $G$ is connected and the minimum degree of $G$ is at least $2$.
We will also assume that for a (partial) $(d_1, d_2)$-coloring $\varphi$ of $G$, the two colors will be $d_1, d_2$ and the graph induced by the color $i$ has maximum degree at most $i$ for $i\in\{d_1, d_2\}$.

In the figures throughout this paper, the white vertices do not have incident edges besides the ones drawn, and the black vertices may have other incident edges. 

\section{Structural Lemmas of $G$}\label{section-lemmas}

In this section, we will reveal some structural aspects of $G$.
A $12^+$-vertex is {\it high} and a $4^-$-vertex is {\it low}. 
A vertex of degree $6$ to $11$ is a {\it medium} vertex. 
Given a (partial) coloring $\varphi:V(G)\rightarrow\{1, 10\}$ of $G$, a vertex $v$ is {\it $i$-saturated} if $\varphi(v)=i$ and $v$ is adjacent to $i$ neighbors colored $i$.
Note that by definition, an $i$-saturated vertex has at least $i$ neighbors. 

\begin{lemma}\label{partial-coloring}
Let $v$ be a $2$-vertex of $G$ where $N(v)=\{v_1, v_2\}$ and $d(v_1)\leq 11$. 
If $\varphi$ is a $(1, 10)$-coloring of $G-v$, then $\varphi(v_1)=1$ and $\varphi(v_2)=10$.
\end{lemma}
\begin{proof}
If $\varphi(v_1)=\varphi(v_2)$, then letting $\varphi(v)\in\{1,10\}\setminus\{\varphi(v_1)\}$ gives a $(1, 10)$-coloring of $G$, which is a contradiction.
If $v_1$ is $10$-saturated and $\varphi(v_2)=1$, then recolor $v_1$ with $1$; this is possible since $d(v_1)\leq 11$.
Now let $\varphi(v)=10$ to obtain a $(1, 10)$-coloring of $G$, which is again a contradiction.
If $v_1$ is not $10$-saturated and $\varphi(v_2)=1$, then let $\varphi(v)=10$ to obtain a $(1, 10)$-coloring of $G$.
Thus, $\varphi(v_2)=10$ and $\varphi(v_2)=1$.
\end{proof}

\begin{lemma}\label{vx-degree}
If $v$ is an $11^-$-vertex of $G$, then $v$ is adjacent to at least one $12^+$-vertex.
\end{lemma}
\begin{proof}
	Suppose that every neighbor of $v$ has degree at most $11$.
	Since $G$ is a counterexample with the minimum number of vertices, $G-v$ has a $(1,10)$-coloring with two colors $1$ and $10$.
	There must exist a neighbor of $v$ that is colored with $1$, otherwise we can color $v$ by $1$ to obtain a $(1, 10)$-coloring of $G$, which is a contradiction. 
	Since each neighbor of $v$ has degree at most $11$, we can recolor each $10$-saturated neighbor of $v$ with the color $1$.
	Now we can color $v$ with the color $10$ since $v$ has no $10$-saturated neighbor and $v$ has at most ten neighbors colored with $10$. 
	This is a $(1, 10)$-coloring of $G$, which is a contradiction. 
\end{proof}

\begin{lemma}\label{no-22}
There are no $2$-vertices adjacent to each other in $G$.
\end{lemma}
\begin{proof}
Assume two $2$-vertices $u, v$ are adjacent to each other and $N(u)=\{v, u'\}$ and $N(v)=\{u, v'\}$. 
Since $G$ is a counterexample with the minimum number of vertices, $G\setminus\{u,v\}$ has a $(1,10)$-coloring $\varphi$ with two colors $1$ and $10$.
If $\varphi(u')=\varphi(v')$, then by letting $\varphi(u)=\varphi(v)\in\{1, 10\}\setminus\{\varphi(v')\}$ we obtain a $(1, 10)$-coloring of $G$, which is a contradiction.
Otherwise, by letting $\varphi(u)\in\{1, 10\}\setminus\{\varphi(u')\}$ and $\varphi(v)\in\{1, 10\}\setminus\{\varphi(v')\}$, we obtain a $(1, 10)$-coloring of $G$, which is also a contradiction.
\end{proof}

Given a vertex $v$ on a face $f$, the {\it $f$-external neighbors} of $v$ are the neighbors of $v$ that are not on $f$. 
A $5$-face is {\it special} if the degrees of the vertices are $2, 12^+, 2, 5, 3$ is some cyclic order. 
Note that by Lemma~\ref{vx-degree}, the $f$-external neighbor of a $3$-vertex on a special face $f$ must be high.

\begin{lemma}\label{special-faces-num}
A vertex $v$ of degree $5$ in $G$ is incident to at most two special faces. 
\end{lemma}
\begin{proof}
Let $v_1, \ldots, v_5$ be the neighbors of $v$ in some cyclic order.
Without loss of generality, assume $v_5$ is a high neighbor of $v$, which is guaranteed by Lemma~\ref{vx-degree}.
This implies that the two faces incident to an edge $vv_5$ cannot be special faces, by the cyclic ordering of the vertices on a special face. 
Suppose that $v$ is incident to three special faces, and let $f$ be the face incident to $v_3v$ and $vv_2$. 
Without loss of generality assume $v_2$ and $v_3$ are a $2$-vertex and a $3$-vertex, respectively.
By Lemma~\ref{vx-degree}, the $f$-external neighbor of $v_3$ must be high.
Yet, the face incident to $vv_3$ that is not $f$ cannot be a special face by the cyclic ordering of the vertices of a special face, which is a contradiction.
\end{proof}

A $5$-face $f$ is an {\it $X_1$-face} if the degrees of the vertices are $2, 12^+, 2, 12^+, 3$ in some cyclic order and the $f$-external neighbor of the $3$-vertex on $f$ is not high.
A $5$-face $f$ is an {\it $X_2$-face} if the degrees of the vertices are $2, 12^+, 2, 12^+, 4$ in some cyclic order and the degrees of the neighbors of the $4$-vertex on $f$ are $11^-, 2, 12^+, 2^+$ in some cyclic order. 
A $5$-face $f$ is a {\it $Y_1$-face} if the degrees of the vertices are $2, 12^+, 2, 4, 3$ in some cyclic order, the degrees of the neighbors of the $4$-vertex on $f$ are $2, 3, 11^-, 12^+$ is some cyclic order, and the faces incident to the two $2$-vertices on $f$ that are not $f$ are an $X_1$-face and an $X_2$-face. 
A $5$-face $f$ is a {\it $Y_2$-face} if the degrees of the vertices are $2, 12^+, 2, 3, 3$ in some cyclic order, 
and the faces incident to the two $2$-vertices on $f$ that are not $f$ are both $X_1$-faces. 
See Figure~\ref{fig-defs-faces}.

A {\it bad} face is a $Y_1$-face or a $Y_2$-face. 
A $5$-face $f$ is a {\it terrible face} if the degrees of the vertices are $2, 12^+, 2, 4, 4$ in some cyclic order, the degrees of the neighbors of both $4$-vertices on $f$ are $2, 4, 11^-, 12^+$ is some cyclic order, and the faces incident to the two $2$-vertices on $f$ that are not $f$ are both $X_2$-faces. 

\begin{figure}[h]
	\begin{center}
  \includegraphics{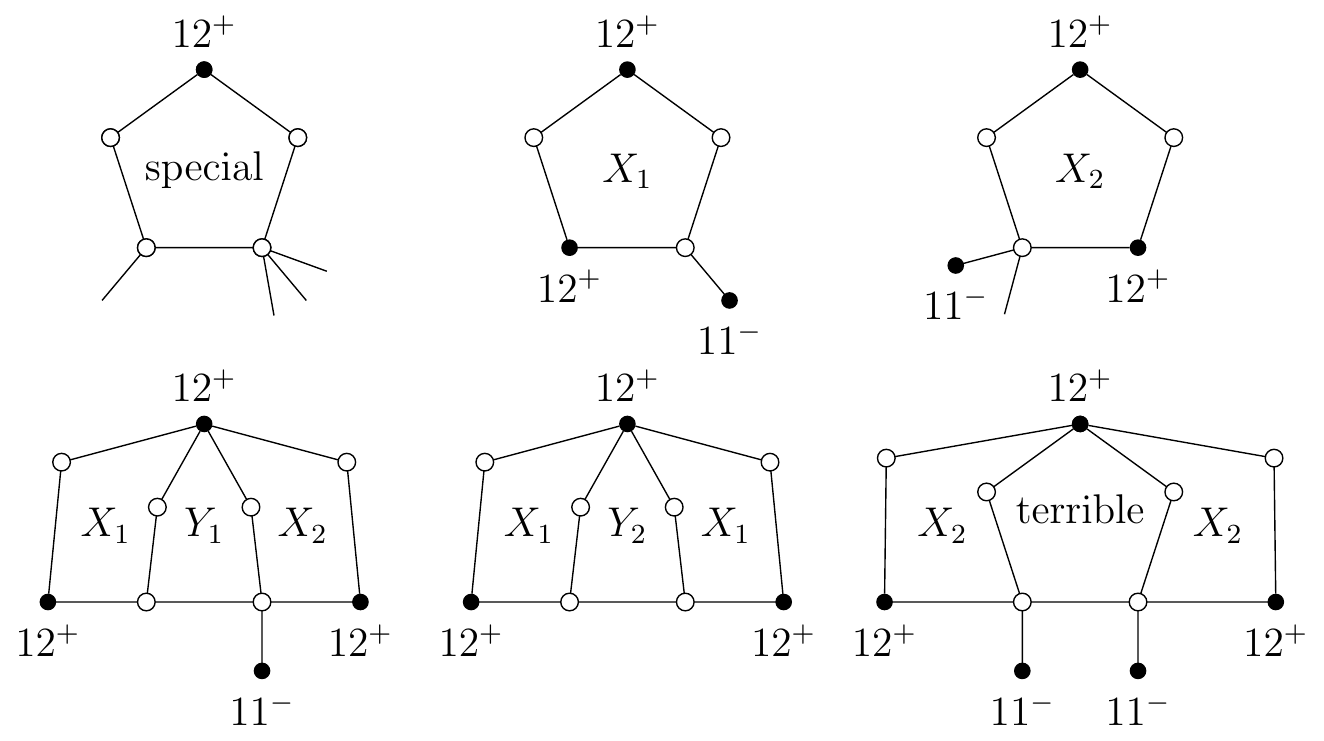}
  \caption{A special face and $X_1$-, $X_2$-, $Y_1$-, $Y_2$-faces and a terrible face.}
  \label{fig-defs-faces}
	\end{center}
\end{figure}

\begin{lemma}\label{terrible-faces-num}
A $12^+$-vertex $v$ of $G$ has at most $\min\{\lfloor {d(v)\over 3}\rfloor, d(v)-12\}$ incident terrible faces. 
\end{lemma}
\begin{proof}
By the configuration of a terrible face, no two terrible faces can share an $X_2$-face. 
Thus, $v$ has at most $\lfloor {d(v)\over 3}\rfloor$ incident terrible faces. 

Suppose $v$ has at least $d(v)-11$ incident terrible faces. 
We will show this implies in some partial coloring $\varphi$ of $G$, $v$ has one uncolored neighbor, ten neighbors colored with $10$, and at least $d(v)-10$ neighbors colored with $1$. 
This contradicts the degree of $v$.

Let $v_1, \ldots, v_{d(v)}$ be the neighbors of $v$ is some cyclic order and for each $2$-vertex $v_i$, let $N(v_i)=\{u_i, v\}$ for $i\in\{1, \ldots, d(v)\}$.
From now on, the addition of the indicies of vertices will be modulo $d(v)$.
Also, for a given terrible face $u_iv_ivv_{i+1}u_{i+1}$ and $j\in\{i, i+1\}$, let $N(u_j)=\{v_j, u_{j-1}, u_{j+1}, w_j\}$. 
Without loss of generality assume $u_4v_4vv_5u_5$ is a terrible face incident to $v$ where the face incident to $v_{1},v,v_{2}$ is not a terrible face. 
Note that such a configuration exists since $3(d(v)-11)=d(v)$ cannot be satisfied. 

Since $G$ is a minimum counterexample, $G-v_4$ has a $(1, 10)$-coloring $\varphi$, and by Lemma~\ref{partial-coloring}, $\varphi(u_4)=1$ and $\varphi(v)=10$.
If $\varphi(u_3)=1$, then $\varphi(w_4)=\varphi(u_5)=10$.
If $w_4$ is $10$-saturated, then recolor $w_4$ with $1$, which is possible since $d(w_4)\leq 11$. 
Now letting $\varphi(v_4)=1$ and recoloring $u_4$ with $10$ gives a $(1, 10)$-coloring of $G$, which is a contradiction. 
Thus $\varphi(u_3)=10$, which implies that $\varphi(v_3)=1$ since otherwise letting $\varphi(v_4)=10$ and recoloring $v_3$ with $1$ gives a $(1, 10)$-coloring of $G$, which is a contradiction.

Similarly, for a terrible face $u_iv_ivv_{i+1}u_{i+1}$, if $\varphi(u_{i-1})=10$, then $\varphi(v_{i-1})=1$, since otherwise letting $\varphi(v_4)=10$ and recoloring $v_{i-1}$ with $1$ gives a $(1, 10)$-coloring of $G$, which is a contradiction.
Suppose $\varphi(u_{i-1})=1$ and $\varphi(v_i)=10$. 
Then letting $\varphi(v_4)=10$ and recoloring $v_i$ with $1$ must not give a $(1, 10)$-coloring of $G$, so $\varphi(u_i)=1$. 
Now if $w_i$ is $10$-saturated, then recolor it with $1$, which is possible since $d(w_i)\leq 11$.
Now letting $\varphi(v_4)=10$ and recoloring $v_i$ and $u_i$ with $1$ and $10$, respectively gives a $(1, 10)$-coloring of $G$, which is a contradiction.
Thus, if $\varphi(u_{i-1})=1$, then $\varphi(v_i)=1$. 
That is, for every terrible face $u_iv_ivv_{i+1}u_{i+1}$, at least one of $\varphi(v_{i-1})$ and $\varphi(v_i)$ is $1$.

Now consider a terrible face $u_jv_jvv_{j+1}u_{j+1}$ that is incident to $v$ with the highest index $j$. 
Since $u_{j+2}\neq u_3$, by the above logic, either $\varphi(v_{j+1})=1$ (if $\varphi(u_{j+2})=1$) or $\varphi(v_{j+2})=1$ (if $\varphi(u_{j+2})=10$). 
This implies that at least two of $v_{j-1}, v_j, v_{j+1}, v_{j+2}$ are colored with $1$.

Since there are at least $d(v)-11$ terrible faces incident to $v$ and none of the vertices that we verified to be colored with $1$ can be counted twice, $v$ has at least $d(v)-10$ neighbors colored with $1$. 
Since coloring $v_4$ with $10$ must not give a $(1, 10)$-coloring of $G$, we know that $v$ is $10$-saturated, which implies that $v$ has ten neighbors colored with $10$. 
Since $v_4$ is uncolored, we showed that $v$ has at least $d(v)-10+10+1=d(v)+1$ neighbors, which is a contradiction.
\end{proof}

\begin{lemma}\label{bad-faces-num}
A $12^+$-vertex $v$ of $G$ has at most $\min\{\lfloor {d(v)\over 3}\rfloor, d(v)-12\}$ incident bad faces. 
\end{lemma}
\begin{proof}
Follows from Lemma~\ref{terrible-faces-num}, since each bad face is a subgraph of a terrible face.
\end{proof}

A face $f_1$ is a {\it $(d_2, d_3)$-sponsor} of an adjacent face $f_2$ if $f_1$ and $f_2$ share the edge $u_2u_3$ where $d(u_2)=d_2, d(u_3)=d_3$ and $u_1,u_2,u_3,u_4$ are consecutive vertices of $f_1$ where $u_1,u_4$ are high vertices. 
A face that is not an $X_1$-face is a {\it sponsor} of another face if both $u_2,u_3$ are low vertices.
See Figure~\ref{fig-sponsor}.

\begin{figure}[h]
	\begin{center}
  \includegraphics{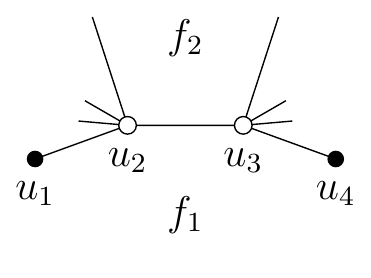}
  \caption{$f_1$ is a $(5,5)$-sponsor of $f_2$}
  \label{fig-sponsor}
	\end{center}
\end{figure}

\section{Discharging Procedure}\label{section-discharging}

Since the embedding of $G$ is fixed, we can let $F(G)$ denote the set of faces of this embedding. 
In this section, we will prove that $G$ cannot exist by assigning an \emph{initial charge} $\mu(z)$ to each $z\in V(G) \cup F(G)$, and then applying a discharging procedure to end up with {\it final charge} $\mu^*(z)$ at $z$.
We prove that the final charge sum is greater $6\gamma -12$, whereas the initial charge sum is equal to $6\gamma-12$.
The discharging procedure will preserve the sum of the initial charge, and hence we find a contradiction to conclude that the counterexample $G$ does not exist.

For each vertex $v\in V(G)$, let $\mu(v)=2d(v)-6$, and for each face $f\in F(G)$, let $\mu(f)=d(f)-6$. 
The initial charge sum is $6\gamma-12$, since
$$\sum_{z\in V(G)\cup F(G)} \mu(z)
	=\sum_{v\in V(G)} (2d(v)-6)+\sum_{f\in F(G)} (d(f)-6) 
	=-6\abs{V(G)}+6\abs{E(G)}-6\abs{F(G)}
	=6\gamma-12$$
The last inequality holds by Euler's formula.

\subsection{Discharging Rules}\label{rules}

The discharging rules (R1)--(R4) indicate how the vertices distribute their initial charge to incident faces. 
Rule (R5) is the only rule where a face sends charge to a vertex, and rules (R6)--(R8) instruct faces on how to send and receive charge between faces. 

Here are the discharging rules:

\begin{enumerate}[(R1)]

\item Each $4$-vertex sends charge ${1\over 2}$ to each incident face.

\item Each $5$-vertex $v$ sends charge ${3\over 2}$ to each incident special face and charge $1$ to each other incident face where a high neighbor of $v$ is not incident to $f$.

\item Each medium vertex $v$ distributes its initial charge uniformly to each incident face $f$ where a high neighbor of $v$ is not incident to $f$.

\item Each high vertex sends charge $2$ to each incident bad face and sends charge $\frac{3}{2}$ to each other incident face. 

\item Each face sends charge $1$ to each incident $2$-vertex. 

\item Each face $f_1$ sends charge ${1}$ to each face $f_2$ where $f_1$ is a $(3, 3)$-, $(3, 4)$-, $(4, 3)$-, or $(4, 4)$-sponsor of $f_2$. 

\item Each face $f_1$ that is not an $X_1$-face sends charge $1\over 2$ to each face $f_2$ where $f_1$ is a $(2, 3)$- or $(3, 2)$-sponsor of $f_2$. 

\item Assume $f_1$ is a $(2, 4)$- or $(4, 2)$-sponsor of $f_2$. 

\begin{enumerate}[(R8A)]

\item If $f_1$ is an $X_2$-face, then $f_1$ sends charge $1\over 2$ to $f_2$. 

\item If $f_1$ is not an $X_2$-face, then $f_1$ sends charge $1$ to $f_2$. 

\end{enumerate}

\end{enumerate}

\begin{figure}[h]
	\begin{center}
  \includegraphics[scale=0.9]{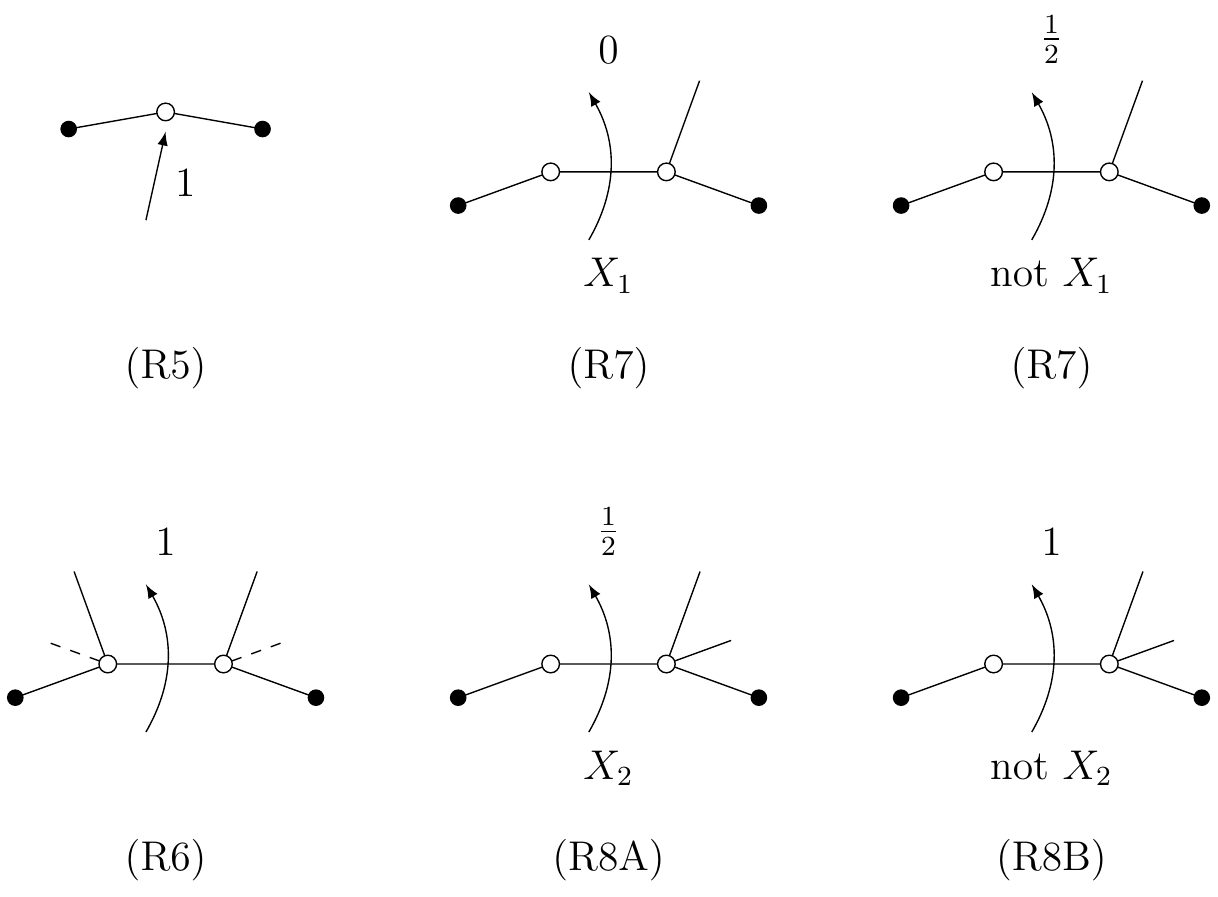}
  \caption{Discharging rules (R5)--(R8)}
  \label{fig-rules}
	\end{center}
\end{figure}

Here is a list of facts to keep in mind:

\begin{itemize}
\item A medium vertex $v$ will send charge either $0$ or at least ${2d(v)-6\over d(v)-2}\geq {3\over 2}$ to each incident face by (R3), since $v$ must have a high neighbor by Lemma~\ref{vx-degree}. 

\item A high vertex will send charge at least ${3\over 2}$ to each incident face by (R4).

\item If (R7) or (R8) happens, then (R5) must happen as well.

\item A face $f$ will spend charge at most $3\over 2$ each time $f$ is a sponsor, except for when (R8B) applies; in this case, $f$ receives charge  $1\over 2$ from the $4$-vertex on $f$, so the net charge sent is still $1+1-{1\over 2}={3\over 2}$.

\end{itemize}

Since (R7) and (R8) include instances of (R5), we will say that an instance of (R5) is {\it independent} if it is not part of an instance of (R7) and (R8).

\subsection{Claims}\label{claims}

We will first show that each vertex has nonnegative final charge.
Then, we will show that each face has nonnegative final charge.

\begin{claim}
Each vertex $v$ has nonnegative final charge.
\end{claim}
\begin{proof}
If $v$ is a $19^+$-vertex, then by Lemma~\ref{bad-faces-num}, 
$\mu^*(v)\geq 2d(v)-6-2\lfloor {d(v)\over 3}\rfloor-{3\over 2}\lceil {2d(v)\over 3}\rceil>0$.
If $12\le d(v) \le 18$, then by Lemma~\ref{bad-faces-num}, $\mu^*(v)\geq 2d(v)-6-2(d(v)-12)-12\cdot{3\over 2}=0$.
%
%
If $v$ is a medium vertex, then by (R3), the final charge is nonnegative since $v$ only distributes its initial charge. 
If $v$ is a $5$-vertex, then by Lemma~\ref{special-faces-num}, $\mu^*(v)\geq 2\cdot5-6-2\cdot{3\over 2}-1=0$.
If $v$ is a $4$-vertex, then $\mu^*(v)=2\cdot 4-6-4\cdot{1\over 2}=0$.
If $v$ is a $3$-vertex, then $v$ neither sends nor receives any charge, so $\mu^*(v)=\mu(v)=2\cdot 3-6=0$.
If $v$ is a $2$-vertex, then by (R5), $v$ receives charge $1$ from each of its incident faces, so $\mu^*(v)=2\cdot 2-6+2\cdot 1=0$.
\end{proof}

\begin{claim}
Each face $f$ of degree at least $7$ has positive final charge.
\end{claim}
\begin{proof}
If $f$ sends charge by (R6), (R7), or (R8A), $f$ receives charge $2\cdot{3\over 2}$ and sends charge at most $3\over 2$ for each instance. 
If $f$ sends charge by (R8B), $f$ receives charge $2\cdot {3\over 2}+{1\over 2}$ and sends charge at most $1+1$ for each instance. 
In either case, $f$ can use half of the charge it receives from the high vertices (not the $4$-vertex when $f$ is a $(2, 4)$- or $(4, 2)$-sponsor) to take care of each instance of (R6)--(R8).

For each independent instance of (R5), one of the two neighbors of the $2$-vertex must be high by Lemma~\ref{vx-degree}, so $f$ receives charge at least ${3\over 2}$. 
Since $f$ may only use half of this ${3\over 2}$ to this instance, $f$ may need an additional $1-{1\over 2}\cdot{3\over 2}={1\over 4}$ charge, which can be provided by the initial face charge of $f$.
By Lemma~\ref{no-22}, there are at most $\lfloor {d(f)\over 2}\rfloor$ $2$-vertices incident with $f$.
If ${d(f)-6\over\lfloor d(f)/2\rfloor}> {1\over 4}$ is satisfied, then we have $\mu^*(f)> 0$, since $f$ may need an additional ${1\over 4}$ charge for each instance of (R5).
If $d(f)\geq 7$, then the aforementioned equation is satisfied.
\end{proof}

\begin{claim}
Each $6$-face $f$ has nonnegative final charge.
\end{claim}
\begin{proof}
Assume $f$ sends charge by (R6)--(R8), which implies that $f$ has at least two high vertices. 
If (R6)--(R8) applies to $f$ twice, then $\mu^*(f)\geq 2\cdot{3\over 2}-2\cdot{3\over 2}=0$.
If (R6)--(R8) applies to $f$ once, then there is at most one independent instance of (R5).
Thus, $\mu^*(f)\geq 2\cdot{3\over 2}-{3\over 2}-1>0$.

Now assume $f$ does not send charge by (R6)--(R8).
If $f$ has three $2$-vertices, then $f$ must have at least two high vertices by Lemma~\ref{no-22}.
Thus, $\mu^*(f)\geq 2\cdot{3\over 2}-3\cdot 1=0$.
If $f$ has one $2$-vertex, then $f$ must have at least one high vertex by Lemma~\ref{no-22}.
Thus, $\mu^*(f)\geq {3\over 2}-1>0$.

Assume $f$ has exactly two $2$-vertices. 
If $f$ has two $2$-vertices and at least two high vertices, then, $\mu^*(f)\geq 2\cdot{3\over 2}-2\cdot 1>0$.
Since $f$ has a $2$-vertex, $f$ must have one high vertex $v$. 
Without loss of generality, assume $v,u_1,u_2,u_3,u_4,u_5$ are vertices of $f$ in cyclic order so that $u_1,u_5$ are $2$-vertices. 
If $u_2, u_3, u_4$ are not all $3$-vertices, then $\mu^*(f)\geq {3\over 2}+{1\over 2}-2\cdot 1=0$.
Otherwise, $d(u_2)=d(u_3)=d(u_4)=3$. 
The $f$-external neighbors of $u_2,u_3,u_4$ must be high by Lemma~\ref{vx-degree}, and therefore $f$ will receive charge at least $1$ by (R6).
Thus, $\mu^*(f)\geq {3\over 2}+1-2\cdot 1>0$.

The only remaining case is when $f$ has no $2$-vertices, which implies that no rules from (R5)--(R8) apply and so $f$ does not send any charge. 
Thus, $\mu^*(f)\ge \mu(f)=6-6=0$.
\end{proof}

\begin{claim}
Each $5$-face $f$ has nonnegative final charge.
\end{claim}
\begin{proof}
Let $v_1, \ldots, v_5$ be the vertices of $f$ in some cyclic order. Note that $\mu(f)=-1$.

Assume $f$ sends charge to another face by (R6)--(R8), which implies that $f$ has two high vertices; without loss of generality assume $v_1$ and $v_4$ are the high vertices, which each sends charge at least ${3\over 2}$ to $f$ by (R4).
Note that $f$ is not an $X_1$-face, which is not a sponsor of another face by definition.
If $v_5$ is not a $2$-vertex, then $\mu^*(f)\geq 2\cdot {3\over 2}-1-{3\over 2}>0$, since the net charge $f$ spends by (R6)--(R8) is at most $3\over 2$.
Therefore we may assume $v_5$ is a $2$-vertex. 
If $f$ sends charge to another face by (R6) or (R8A), then $\mu^*(f)\geq 2\cdot {3\over 2}-1-1-1=0$.

If $f$ sends charge to another face $f'$ by (R7), then $f$ must be either a $(2, 3)$-sponsor or a $(3, 2)$-sponsor.
Without loss of generality assume $v_2$ is the $3$-vertex on $f$ and let $u$ be the $f$-external neighbor of $v_2$.
Since $f$ is not an $X_1$-face, $u$ must be high.
This implies that $f'$ cannot be an $X_1$-face since $v_1$ is high, so $f'$ will also send charge ${1\over 2}$ to $f$ as well.
Thus, $\mu^*(f)\geq 2\cdot {3\over 2}-1-1-1-{1\over 2}+{1\over 2}=0$. 

If $f$ sends charge to another face $f'$ by (R8B), then $f$ is not an $X_2$-face.
Without loss of generality assume $v_2$ is the $4$-vertex on $f$ and let $u$ be the $f$-external neighbor of $v_2$ on $f'$. 
Now $u$ must be a high vertex otherwise $f$ becomes an $X_2$-face, which implies that $f'$ is either a $(2, 4)$-sponsor or a $(4, 2)$-sponsor of $f$ as well.
Thus, $\mu^*(f)\geq 2\cdot {3\over 2}+{1\over 2}-1-1-1-1+{1\over 2}=0$.

\medskip

Now assume $f$ is not a sponsor of another face. 

Assume $f$ is incident to no $2$-vertices. 
If $f$ is incident to a high vertex, then $\mu^*(f)\geq {3\over 2}-1>0$.
If $f$ is incident to a $5^+$-vertex and no high vertex, then $\mu^*(f)\geq 1-1=0$.
If $f$ is incident to at least two $4$-vertices, then $\mu^*(f)\geq 2\cdot{1\over 2}-1=0$.
Otherwise, $f$ has two consecutive $3$-vertices, and their $f$-external neighbors must be both high. 
Thus, $\mu^*(f)\geq -1+1=0$, since $f$ will receive charge $1$ by (R6).

It is easy to see that the number of $2$-vertices on $f$ is at most $2$.
Assume $f$ is incident to exactly one $2$-vertex $v_2$ and without loss of generality assume $v_1$ is a high vertex, which exists by Lemma~\ref{vx-degree}.
If $f$ is incident to two high vertices, then $\mu^*(f)\geq 2\cdot{3\over 2}-1-1>0$.
Otherwise $f$ is incident to exactly one high vertex.
If either $v_3$ or $v_4$ is a $4^+$-vertex, then $\mu^*(f)\geq {3\over 2}+{1\over 2}-1-1=0$.
If both $v_3$ and $v_4$ are $3$-vertices, then their $f$-external neighbors must be both high.
Thus, $\mu^*(f)\geq {3\over 2}-1-1+1>0$, since $f$ will receive charge $1$ by (R6).

Now assume $f$ is incident to exactly two $2$-vertices. 
If $f$ is also incident to at least two high vertices, then $\mu^*(f)\geq 2\cdot{3\over 2}-1-2\cdot 1=0$.
Since a $2$-vertex must be adjacent to a high vertex by Lemma~\ref{vx-degree}, we may assume that $v_1$ and $v_3$ are the $2$-vertices and $v_2$ is the only high vertex on $f$.
If either $v_4$ or $v_5$ is a $6^+$-vertex, then $\mu^*(f)\geq {3\over 2}+{3\over 2}-1-2\cdot 1=0$.
We may assume $d(v_4)\leq d(v_5)$.
If $v_5$ is a $5$-vertex and $v_4$ is a $4^+$-vertex, then $\mu^*(f)\geq {3\over 2}+1+{1\over 2}-1-2\cdot 1=0$.
If $v_5$ is a $5$-vertex and $v_4$ is a $3$-vertex,  then $f$ is a special face and $\mu^*(f)\geq {3\over 2}+{3\over 2}-1-2\cdot 1=0$, since $v_5$ now sends charge ${3\over 2}$ to $f$ by (R2).
If $v_5$ and $v_4$ are $4$-vertices, then $\mu^*(f)\geq {3\over 2}+2\cdot{1\over 2}+{1\over 2}-1-2\cdot 1=0$, since at least one adjacent face to $f$ will be a sponsor for $f$.

If $v_5$ is a $4$-vertex and $v_4$ is a $3$-vertex, then the $f$-external neighbor of $v_4$ must be high by Lemma~\ref{vx-degree}.
If $f$ is a bad face, then $\mu^*(f)\geq 2+{1\over 2}-1-2\cdot 1+{1\over 2}=0$, since $v_2$ now sends charge $2$ to $f$ by (R4) and $f$ receives charge ${1\over 2}$ by (R8A).
If $f$ is not a bad face, then either the face incident to $v_1$ that is not $f$ is not an $X_2$-face or the face incident to $v_3$ that is not $f$ is not an $X_1$-face. 
If the $f$-external neighbor of $v_5$ that is on the same face as the edge $v_4v_5$ is high, then $\mu^*(f)\geq {3\over 2}+{1\over 2}-1-2\cdot 1+1=0$, since $f$ will receive charge $1$ by (R6).
Otherwise, either $\mu^*(f)\geq {3\over 2}+{1\over 2}-1-2\cdot 1+2\cdot{1\over 2}=0$, since $f$ will receive charge ${1\over 2}$, ${1\over 2}$ by (R7), (R8A), respectively, or $\mu^*(f)\geq {3\over 2}+{1\over 2}-1-2\cdot 1+1=0$, since $f$ will receive charge $1$ by (R8B).
See Figure~\ref{fig-helper}.

\begin{figure}[h]
	\begin{center}
	\includegraphics[scale=0.8]{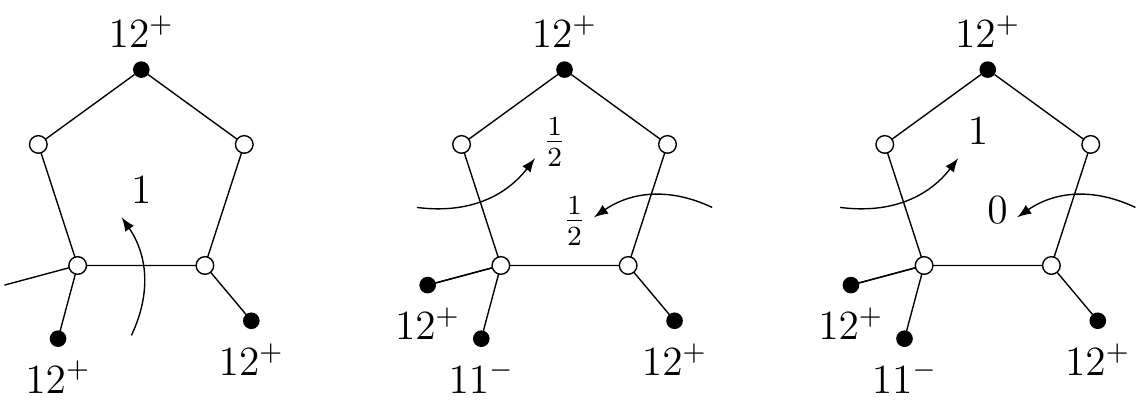}
	\caption{When $d(v_4)=3$ and $d(v_5)=4$}
	\label{fig-helper}
	\end{center}
\end{figure}

If $v_5$ and $v_4$ are both $3$-vertices, then the $f$-external neighbors of $v_5$ and $v_4$ must be high by Lemma~\ref{vx-degree}. 
If $f$ is a bad face, then $\mu^*(f)\geq 2-1-2\cdot 1+{1}=0$, since $v_2$ now sends charge $2$ to $f$ by (R4) and $f$ will receive charge $1$ by (R6).
If $f$ is not a bad face, then the face incident to either $v_1$ or $v_3$ that is not $f$ is not an $X_1$-face. 
Thus, $\mu^*(f)\geq {3\over 2}-1-2\cdot 1+1+{1\over 2}=0$, since $f$ will receive charge $1$, ${1\over 2}$ by (R6), (R7), respectively.
\end{proof}

\section{Proofs of Theorems}\label{section-proof}

We finish the paper by proving Theorem~\ref{thm-gen}, which implies Theorem~\ref{thm-main} and Corollary~\ref{cor-planar}. 
Note that the lemmas used in Section~\ref{section-lemmas} are specifically for $\gamma\leq 1$. 
However, in order to improve the readability of the paper, we did not explicitly rewrite all the lemmas for when $\gamma\geq 2$ as the ideas are identical and adds no extra value. 
We restate the generalized version of Lemma~\ref{vx-degree} and Lemma~\ref{no-22} in order to prove Lemma~\ref{vx-high-general}, which is used in the proof of Theorem~\ref{thm-gen}.

\begin{lemma}\label{vx-degree-gen}
Let $H$ be a graph with girth at least $5$ that is not $(1,t)$-colorable but every proper subgraph of $H$ is $(1, t)$-colorable.
\begin{enumerate}[$(i)$]
\item If $v$ is a $(t+1)^-$-vertex of $H$, then $v$ is adjacent to at least one $(t+2)^+$-vertex.
\item There are no $2$-vertices adjacent to each other in $H$.
\item The minimum degree of $H$ is at least $2$.
\end{enumerate}
\end{lemma}

\begin{lemma}\label{vx-high-general}
If $H$ is a graph with girth at least $5$ that is not $(1,t)$-colorable but every proper subgraph of $H$ is $(1, t)$-colorable, then $H$ has at least three $(t+2)^+$-vertices. 
\end{lemma}
\begin{proof}
Since $H$ cannot be a tree, which is $(0, 0)$-colorable, there must exist some cycle $C$ in $H$. 
Since $H$ has girth at least $5$, $C$ has at least five vertices. 
If $H$ has at most two $(t+2)^+$-vertices, then there exists an edge $uv$ on $C$ where $d(u), d(v)\leq t+1$. 

By Lemma~\ref{vx-degree-gen} $(i)$, each of $u, v$ has a neighbor $u'$, $v'$, respectively, that is a $(t+2)^+$-vertex.
Note that $u'\neq v'$ since otherwise $H$ has a cycle $u, v, u'=v'$ of length $3$.
Since the minimum degree of $H$ is at least $2$ and two $2$-vertices cannot be adjacent by Lemma~\ref{vx-degree-gen} $(iii)$, $(ii)$, we may assume that $u$ has a neighbor $z$ that is neither $v$ nor $u'$. 
If $z$ is a $(t+2)^+$-vertex, then $u', v', z$ are three $(t+2)^+$-vertices of $H$.

If $z$ is not a $(t+2)^+$-vertex, then $z$ must have a neighbor $z'$ of degree at least $t+2$ by Lemma~\ref{vx-degree-gen} $(i)$.
Note that $z'\not\in\{u', v'\}$ since otherwise $H$ has a cycle of length at most $4$. 
Hence, $H$ has at least three $(t+2)^+$-vertices $u', v', z'$.
\end{proof}

\newtheorem*{thm-general}{Theorem \ref{thm-gen}}

\begin{thm-general}
Given a surface $S$ of Euler genus $\gamma$, every graph with girth at least $5$ that is embeddable on $S$ is $\big(1, K(\gamma)\big)$-colorable where $K(\gamma)=\max\{10, 4\gamma+3\}$.
\end{thm-general}
\begin{proof}
Suppose a counterexample to the theorem exists, and consider the minimum one.

If $S$ is the plane or the projective plane, then $\gamma\leq 1$ and so $K(\gamma)=10$.
The initial charge sum is less than $0$, but by claims in Subsection~\ref{claims}, the sum of the final charge is nonnegative.
Thus, we conclude that there is no such counterexample. 

If $S$ is neither the plane nor the projective plane, then $K(\gamma)=4\gamma+3$ and let $k(v, \gamma)=\min\left\{\lfloor{d(v)\over 3}\rfloor, d(v)-K(\gamma)-2\right\}$. 
By Lemma~\ref{bad-faces-num} and (R4) of the discharging procedure in Section~\ref{section-discharging}, for a $(K(\gamma)+2)^+$-vertex $v$,
$\mu^*(v)
\geq 2d(v)-6-2k(v, \gamma)-{3\over 2}\left(d(v)-k(v, \gamma)\right)
={1\over 2}\left(d(v)-k(v,\gamma)\right)-6$.
If $d(v)-\lfloor{d(v)\over 3}\rfloor\leq K(\gamma)+2$, then $k(v, \gamma)=d(v)-K(\gamma)-2$, so $\mu^*(v)\geq{K(\gamma)+2\over 2}-6$.
Otherwise, $d(v)-\lfloor{d(v)\over 3}\rfloor> K(\gamma)+2$, and $k(v, \gamma)=\lfloor{d(v)\over 3}\rfloor$, so $\mu^*(v)\geq{1\over 2}(d(v)-\lfloor{d(v)\over 3}\rfloor)-6> {K(\gamma)+2\over 2}-6$.
Therefore, $\mu^*(v)\geq {K(\gamma)\over 2}-5=2\gamma-3.5$.

Note that Lemma~\ref{vx-high-general} guarantees the existence of three $(K(\gamma)+2)^+$-vertices.
Thus, via the discharging procedure explained in Section~\ref{section-discharging} and the claims in Subsection~\ref{claims}, the final charge sum is at least $3(2\gamma-3.5)$, which is strictly greater than the initial charge sum, which is $6\gamma-12$.
This is a contradiction, and therefore there is no counterexample to the theorem. 
\end{proof}

\section{Acknowledgments}

The authors thank Andr\'e Raspaud for introducing the problem to the second author when he was a student at the University of Illinois at Urbana--Champaign.

\bibliographystyle{plain}
\bibliography{PlanarGirth5_110color}

\begin{thebibliography}{10}

\bibitem{1977ApHa}
K.~Appel and W.~Haken.
\newblock Every planar map is four colorable. {I}. {D}ischarging.
\newblock {\em Illinois J. Math.}, 21(3):429--490, 1977.

\bibitem{1977ApHaKo}
K.~Appel, W.~Haken, and J.~Koch.
\newblock Every planar map is four colorable. {II}. {R}educibility.
\newblock {\em Illinois J. Math.}, 21(3):491--567, 1977.

\bibitem{2010BoIvMoOcRa}
O.~V. Borodin, A.~O. Ivanova, M.~Montassier, P.~Ochem, and A.~Raspaud.
\newblock Vertex decompositions of sparse graphs into an edgeless subgraph and
  a subgraph of maximum degree at most {$k$}.
\newblock {\em J. Graph Theory}, 65(2):83--93, 2010.

\bibitem{2011BoIvMoRa}
O.~V. Borodin, A.~O. Ivanova, M.~Montassier, and A.~Raspaud.
\newblock {$(k,j)$}-coloring of sparse graphs.
\newblock {\em Discrete Appl. Math.}, 159(17):1947--1953, 2011.

\bibitem{2012BoIvMoRa}
O.~V. Borodin, A.~O. Ivanova, M.~Montassier, and A.~Raspaud.
\newblock {$(k,1)$}-coloring of sparse graphs.
\newblock {\em Discrete Math.}, 312(6):1128--1135, 2012.

\bibitem{2013BoKoYa}
O.~V. Borodin, A.~Kostochka, and M.~Yancey.
\newblock On 1-improper 2-coloring of sparse graphs.
\newblock {\em Discrete Math.}, 313(22):2638--2649, 2013.

\bibitem{2011BoKo}
O.~V. Borodin and A.~V. Kostochka.
\newblock Vertex decompositions of sparse graphs into an independent set and a
  subgraph of maximum degree at most 1.
\newblock {\em Sibirsk. Mat. Zh.}, 52(5):1004--1010, 2011.

\bibitem{2014BoKo}
O.~V. Borodin and A.~V. Kostochka.
\newblock Defective 2-colorings of sparse graphs.
\newblock {\em J. Combin. Theory Ser. B}, 104:72--80, 2014.

\bibitem{2009BoIv}
O.V. Borodin and A.O. Ivanova.
\newblock Near proper $2$-coloring the vertices of sparse graphs.
\newblock Diskretn. Anal. Issled. Oper., 16(2):16-20, 2009.

\bibitem{unpub_ChRa}
I.~Choi and A.~Raspaud.
\newblock {Planar graphs with minimum cycle length at least $5$ are $(3,
  5)$-colorable}.
\newblock Submitted, 2013.

\bibitem{1986CoCoWo}
L.~J. Cowen, R.~H. Cowen, and D.~R. Woodall.
\newblock Defective colorings of graphs in surfaces: partitions into subgraphs
  of bounded valency.
\newblock {\em J. Graph Theory}, 10(2):187--195, 1986.

\bibitem{1999EaHu}
Nancy Eaton and Thomas Hull.
\newblock Defective list colorings of planar graphs.
\newblock {\em Bull. Inst. Combin. Appl.}, 25:79--87, 1999.

\bibitem{2013EsMoOcPi}
Louis Esperet, Micka{\"e}l Montassier, Pascal Ochem, and Alexandre Pinlou.
\newblock A complexity dichotomy for the coloring of sparse graphs.
\newblock {\em J. Graph Theory}, 73(1):85--102, 2013.

\bibitem{2007GlZa}
A.~N. Glebov and D.~Zh. Zambalaeva.
\newblock Path partitions of planar graphs.
\newblock {\em Sib. \`Elektron. Mat. Izv.}, 4:450--459, 2007.

\bibitem{2006HaSe}
Fr{\'e}d{\'e}ric Havet and Jean-S{\'e}bastien Sereni.
\newblock Improper choosability of graphs and maximum average degree.
\newblock {\em J. Graph Theory}, 52(3):181--199, 2006.

\bibitem{unpub_KiKoZh}
Jaehoon Kim, Alexandr~V. Kostochka, and Xuding Zhu.
\newblock Private communication, 2013.

\bibitem{unpub_MoOc}
M.~Montassier and P.~Ochem.
\newblock Near-colorings: non-colorable graphs and np-completeness.
\newblock Submitted, 2013.

\bibitem{1999Sk}
R.~{\v{S}}krekovski.
\newblock List improper colourings of planar graphs.
\newblock {\em Combin. Probab. Comput.}, 8(3):293--299, 1999.

\bibitem{2000Sk}
Riste {\v{S}}krekovski.
\newblock List improper colorings of planar graphs with prescribed girth.
\newblock {\em Discrete Math.}, 214(1-3):221--233, 2000.

\end{thebibliography}

%

\end{document}